\documentclass[12pt]{amsart}
\usepackage{amsfonts}
\usepackage[cp1250]{inputenc}
\usepackage{amssymb}
\usepackage{amsmath}
\usepackage{amsthm}
\usepackage{latexsym}
\usepackage{amsbsy}
\usepackage{graphicx}
\usepackage{color}
\usepackage{boxedminipage}
\usepackage{geometry}
\usepackage[cp1250]{inputenc}
\usepackage{graphicx}
\usepackage{cite}

\newtheorem{theorem}{Theorem}
\newtheorem*{theoremH}{Helson's Theorem}
\newtheorem*{theoremI}{Hitt's Theorem}
\newtheorem*{theoremY}{Hayashi's Theorem}
\newtheorem*{theoremS}{Sarason's Theorem}
\theoremstyle{plain}

\newtheorem*{lem}{Lemma}

\newtheorem*{wn}{Corollary}

\theoremstyle{definition}

\newtheorem*{rem}{Remark}
\theoremstyle{definition}
\newtheorem*{ex}{Example}

\DeclareMathOperator{\im}{Im}

\def\C{\mathbb{C}}%
\def\D{\mathbb{D}}%
\def\N{\mathbb{N}}%

\def\H{\mathcal{H}}%
\def\M{\mathcal{M}}%
\def\P{\mathcal{P}}%

\date{\today}

\begin{document}
\title{On kernels of Toeplitz operators}
\author{M. T. Nowak, P. Sobolewski, A. So\l tysiak and M. Wo\l oszkiewicz-Cyll}

\address{Maria T. Nowak,  Pawe{\l} Sobolewski, Magdalena Wo{\l}oszkiewicz-Cyll
\newline Institute of Mathematics
\newline Maria Curie-Sk{\l}odowska University \newline pl. M.
Curie-Sk{\l}odowskiej 1 \newline 20-031 Lublin, Poland}
\email{mt.nowak@poczta.umcs.lublin.pl,
pawel.sobolewski@poczta.umcs.lublin.pl,
woloszkiewicz@poczta.umcs.lublin.pl}

\address{Andrzej So{\l}tysiak  \newline Faculty of Mathematics and Computer
Science
\newline Adam Mickiewicz University \newline ul. Uniwersytetu Pozna\'nskiego 4 \newline 61-614 Pozna\'n, Poland}
\email{asoltys@amu.edu.pl}

\subjclass [2010]{47B32, 46E22, 30H10} \keywords{Toeplitz operators,
de Branges-Rovnyak spaces, nearly invariant subspaces, rigid
functions, nonextreme functions, kernel functions}

\begin{abstract} We apply the theory of de Branges-Rovnyak spaces  to
describe kernels of some Toeplitz operators on the classical Hardy
space $H^2$. In particular, we discuss the kernels of the operators
$T_{\bar f/ f}$ and $T_{\bar I\bar f/ f}$, where $f$  is an outer
function in $H^2$ and $I$ is inner such that $I(0)=0$. We also
obtain results on  de Branges-Rovnyak spaces generated by nonextreme
functions.
\end{abstract}

\maketitle

\section{Introduction}
Let $H^2$ denote the standard Hardy space on the unit disk $\D$. For
$\varphi\in L^{\infty}(\partial\D)$ the Toeplitz operator on $H^2$
is given by $T_{\varphi}f=P_{+}(\varphi f)$, where $P_{+}$ is the
orthogonal projection of $L^2(\partial\D)$ onto $H^2$. We will
denote by $\M(\varphi)$ the range of $T_{\varphi}$ equipped with the
range norm, that is, the norm that makes the operator $T_{\varphi}$
a coisometry of $H^2$ onto $\M(\varphi)$. For a nonconstant function
$b$ in the unit ball of $H^\infty$ the de Branges-Rovnyak space
$\H(b)$  is the image of  $H^2$ under  the operator
$(1-T_bT_{\bar{b}})^{1/2}$ with the corresponding range norm. The
norm and the inner product in $\H(b)$ will be denoted by
$\|\cdot\|_b$ and $\langle \cdot,\cdot\rangle_b$. The space $\H(b)$
is a Hilbert space with the reproducing kernel
\[
k_w^b(z)=\frac{1-\overline{b(w)}b(z)}{1-\overline{w}z}\quad(z,w\in\mathbb{D}).
\]
In the case when $b$ is an inner function the space $\H(b)$ is the
well-known model space $K_b=H^2\ominus bH^2$.

If the function $b$ fails to be an extreme point of the unit ball in
$H^{\infty}$, that is, when $\log(1-|b|)\in L^1(\mathbb{T})$, we
will say simply  that $b$ is nonextreme. In this case one can define
an outer function $a$ whose modulus on $\partial\D$ equals
$\left(1-|b|^2\right)^{1/2}$. Then we say  that the functions $b$
and $a$ form a \emph{pair} $(b,a)$. By the Herglotz representation
theorem there exists a positive measure $\mu$ on $\partial\D$ such
that
\begin{equation}
\frac{1+b(z)}{1-b(z)}=\int_{\partial\D}\frac {1+ e^{-i\theta}z}{1-
e^{-i\theta}z}\,d\mu(e^{i\theta})+i\im\frac{1+b(0)}{1-b(0)},\quad
z\in\D. \label{Herglotz}
\end{equation}
Moreover the function $\left|\frac{a}{1-b}\right|^2$ is the
Radon-Nikodym derivative of the absolutely continuous component of
$\mu$ with respect to the normalized Lebesgue measure. If the
measure $\mu$ is absolutely continuous the pair $(b,a)$ is called
\emph{special}.

Recall that a function  $f\in H^1$ is called \emph{rigid} if and
only if no other functions in $H^1$, except for positive scalar
multiples  of $f$ have the same argument as $f$ a.e. on
$\partial\D$.

If $(b,a)$ is a pair, then  $\M(a)$ is contained contractively in
$\H(b)$. If a pair $(b,a)$ is special and $f=\frac{a}{1-b}$, then
$\M(a)$ is dense in $\H(b)$ if and only if $f^2$ is rigid
(\cite{Sarason3}). Spaces $\mathcal H(b)$ for nonextreme $b$ have
been studied in \cite{CGR}, \cite{CR}, \cite{FHR1}, \cite{LN},
\cite{LMT}, \cite{SarasonL}, and \cite{Sarason5}.

It is known that the kernel of a Toeplitz operator $T_{\varphi} $ is
a subspace of $H^2$ of the form $\ker T_{\varphi}= fK_I$, where
$K_I= H^2\ominus IH^2$ is the model space corresponding to the inner
function $I$ such that $I(0)=0$ and $f$ is an outer function of unit
$H^2$ norm that acts as an isometric multiplier from  $K_I$ onto $f
K_{I}$. Moreover, $f$ can be expressed as $f=\frac{a}{1-Ib_0}$,
where $(b_0,a)$ is a special pair and $\bigl(\frac{a}{1-b_0}\bigr)^2
$ is a rigid function  in $H^1$. Then we also have   $\ker
T_{\frac{\bar{I}\bar{f}}{f}}=fK_{I}$. In the recent paper
\cite{FHR2} the authors considered  the Toeplitz operator
$T_{\frac{\bar{g}}{g}}$ where $g\in H^{\infty}$ is outer. Among
other results, they described all outer functions $g$ such that $
\ker T_{\frac{\bar{g}}{g}}=K_{I}$. In Section 2 we describe all such
functions $g$ for which $\ker T_{\frac{\bar{g}}{g}}=f K_{I}$.

If   $(b,a)$ is a special pair,   $f=\frac a{1-b}$ and $b=Ib_0$,
where $I$ as above, then $f K_{I}\subset\ker
T_{\frac{\bar{I}\bar{f}}{f}}$. In the next two sections we study the
space $\ker T_{\frac{\bar{I}\bar{f}}{f}}\ominus f K_{I}$ and show
that it is isometrically isomorphic to the orthogonal complement of
$\M(a)$ in the de Branges-Rovnyak space $\H(b_0)$. We also give an
example of a function $f$ for which the space $\ker
T_{\frac{\bar{I}\bar{f}}{f}}\ominus f K_{I}$ is one dimensional. In
the last section we discuss  the orthogonal complement of $\mathcal
{M}(a)$ in $\mathcal {H}(b)$ and  get a  generalization of results
obtained in \cite{LN} and \cite{FHR1} for the case when pairs are
rational.

\section{The kernel of $T_{\frac{\bar g}{g}}$}
It is known that if $g$ is an outer function in $H^2$, then the
kernel of $T_{\frac{\bar{g}}{g}}$ is trivial if and only if $g^2$ is
rigid (see e.g. \cite{Sarason1}).

The finite dimensional kernels of Toeplitz operators were described
by Nakazi \cite{Nakazi}. Nakazi's theorem says that $\dim \ker
T_{\varphi}=n$ if and only if there exists an outer function $f\in
H^2$ such that $f^2$ is rigid and $\ker T_{\varphi}=\{fp\colon\,
p\in\P_{n-1}\},$ where $\P_{n-1}$ denotes the set of all polynomials
of degree at most $n-1$.

Consider the following example.
\begin{ex}
For $\alpha>-\frac{1}{2}$ set $g(z)=(1-z)^\alpha$, $z\in\D$. Then
the kernel of $T_{\frac{\bar{g}}{g}}$ is trivial for $\alpha\in
(-\frac{1}{2},\frac{1}{2}]$ and  dimension of the kernel of
$T_{\frac{\overline{g}}{g}}$ is $n$ for $\alpha\in(n-\frac
12,n+\frac{1}{2}], n=1,2,\ldots $, and
\[
\ker T_{\frac{\overline{(1-z)^\alpha}}{(1-z)^\alpha}}=(1-z)^{\alpha-n}K_{z^n}.
\]
\end{ex}
In the general case the kernels of Toeplitz operators are
characterized by Hayashi's theorem. To state this theorem we need
some notation. We note that an outer function $f$ having unit norm
in $H^2$  $(\|f\|_2=1)$ can be written as
\[
f=\frac{a}{1-b},
\]
where $a$ is an outer function,  $b$ is a function from the unit
ball of $H^{\infty}$ such that $|a|^2+|b|^2=1$ a.e. on $\partial\D$.
Following Sarason \cite[p. 156] {Sarason3} we call $(b,a)$ the
\emph{pair associated with} $f$. Note  also that $b$ is a nonextreme
point of the closed unit ball of $H^{\infty}$ and is given by
\begin{equation}
\frac{1+b(z)}{1-b(z)}=\frac {1}{2\pi}\int_{0}^{2\pi}\frac {1+
e^{-i\theta}z}{1- e^{-i\theta}z}|f(e^{i\theta})|^2d\theta,\quad
z\in\D. \label {outer}
\end{equation}

Let $S$ denote the unilateral shift operator on $H^2$, i.e.
$S=T_z$.
A closed subspace $M$ of $H^2$ is said to be \emph{nearly
$S^\ast$-invariant} if for every $f\in M$ vanishing at 0, we also
have $S^\ast f\in M$. It is known that the kernels of Toeplitz
operators are nearly $S^\ast$-invariant.

Nearly $S^\ast$-invariant spaces are characterized by Hitt's theorem
\cite{Hitt}.

\begin{theoremI}
The closed subspace M of $H^2$ is nearly $S^\ast$-invariant  if and
only if there exists a function f of unit norm  and a model space
$K_I=H^2\ominus IH^2$ such that $M= T_fK_I$, where $I$ is an inner
function vanishing  at the origin, and $T_f$ acts isometrically on
$K_I$.
\end{theoremI}

It has been proved by D. Sarason \cite{Sarason1} that $T_f$ acts
isometrically on $K_I$ if and only if $I$ divides $b$ (the first
function in the pair associated with $f$). Consequently, the
function $f$ in Hitt's theorem can be written as
\[
f=\frac {a}{1-Ib_0}.
\]
The function $ \frac{1+b_0(z)}{1-b_0(z)}$ has a positive real part
and is the Herglotz integral of a positive measure on $\partial\D$
up to an additive imaginary constant,
\begin{equation}
\frac{1+b_0(z)}{1-b_0(z)}=\int_{\partial\D}\frac {1+
e^{-i\theta}z}{1- e^{-i\theta}z}d\mu(e^{i\theta})+ic.\label{Her}
\end{equation}
Clearly $b_0$ is also a nonextreme point of the closed unit ball of
$H^{\infty}$ and $|a|^2+|b_0|^2=1$ a.e. on $\partial\D$.

We remark that in view of (\ref{outer}) the pair   $(b,a)$
associated with an outer function $f\in H^2$ is special, while the
pair $(b_0,a)$ need not to be special. Under the above notations
Hayashi's theorem reads as follows:

\begin{theoremY}
The nearly $S^\ast$-invariant space   $M=T_fK_I$ is the kernel of a
Toeplitz operator if and only if the pair $(b_0,a)$ is special and
$f_0^2=  \Bigl(\dfrac a{1-b_0}\Bigr)^2$ is a rigid function.
\end{theoremY}

Moreover, it follows from Sarason's proof of Hayashi's theorem that
if $M=T_fK_I$ is the kernel of a Toeplitz operator then it is the
kernel of $T_{\frac{\bar{I}\bar f}{f}}$.

Recently E. Fricain, A. Hartmann and W. T. Ross \cite{FHR2}
considered  the Toeplitz operators $T_{\frac{\bar{g}}{g}} $ where
$g\in H^{\infty}$ is outer. If  $\ker T_{\frac{\bar{g}}{g}}$ is
non-trivial, then by Hayashi's theorem there exist the outer
function $f$ and the inner function $I$, $I(0)=0$,   such that
\[
\ker T_{\frac{\bar{g}}{g}}=fK_{I}.
\]
In the above mentioned paper \cite{FHR2} the authors described all
outer functions $g\in H^{\infty}$ for which
\[
\ker T_{\frac{\bar{g}}{g}}=K_I,
\]
where $I$ is an inner function  not necessarily satisfying $I(0)=0$.

We prove the following
\begin{theorem}
Assume that $g\in H^2$ is outer and $M= T_fK_{I}$ is the nearly $S^\ast$-invariant
space, where $I$ is an inner function such that
$I(0)=0$, $(b_0I,a)$  is the pair associated with the outer function
$f$, $(b_0,a)$ is special, and
$f_0^2=\Bigl(\dfrac{a}{1-b_0}\Bigr)^2$ is rigid. Then $\ker
T_{\frac{\bar g}{g}}=M$ if and only if
\[
g= i\frac{I_1+I_2}{I_1-I_2}(1+ I)f,
\]
where $I_1$ and $I_2$ are  inner and $I_1-I_2$ is outer.
\end{theorem}

Recall that  the \emph{Smirnov class} $\mathcal{N}^{+}$ consists of
those holomorphic functions in $\D$ that are quotients of functions
in $H^{\infty}$ in which the denominators are outer functions.

In the proof of Theorem 1, similarly to \cite{FHR2}, we use the
following result due to H. Helson \cite{Helson}.

\begin{theoremH} The functions  $f\in \mathcal{N}^{+}$ that are
real almost everywhere on $\partial\D$ can be written as
\[
f=i\frac{I_1+I_2}{I_1-I_2},
\]
where $I_1$ and $I_2$ are  inner and $I_1-I_2$ is outer.
\end{theoremH}

We also apply a description of kernels in terms of
$S^\ast$-invariant subspaces $K_I^p(|f|^p)$ of weighted Hardy spaces
(in the case when $p=2$) considered by A. Hartmann and K. Seip in
their paper \cite{HS} (see also \cite{H}). For an outer function $f$
in $H^2$ the weighted Hardy space is defined as follows
\[
H^2(|f|^2)=\{g\in \mathcal{N}^{+}\colon\,
\|g\|^2_{2,f}=\frac1{2\pi}\int_0^{2\pi}|g(e^{it})|^2|f(e^{it})|^2dt<\infty\}
\]
and, for an inner function $I$, $K^2_I(|f|^2)= K_{I}(|f|^2)$ is
given by
\[
K_{I}(|f|^2)=\{g=I\overline{\psi}\in H^2(|f|^2)\colon\, \psi \in
H_0^2(|f|^2)\},
\]
where $H_0^2(|f|^2)=zH^2(|f|^2)$.

Then $K_{I}(|f|^2)$ is $S^\ast$-invariant and $fK_{I}(|f|^2)=\ker
T_{\frac{\bar I \bar f}{f}}$ (see \cite{HS}).

\noindent \textit{Proof of Theorem 1.} Assume that $\ker
T_{\frac{\bar g}{g}}=fK_{I}$. Then
\[
fK_I=\ker T_{\frac{\bar I \bar f}{f}}=\ker T_{\frac{\bar g}{g}}.
\]
Since $f\in\ker T_{\frac{\bar I \bar f}{f}}$, the last equalities
imply that
\[
\frac{\overline{g}f}{g}= \overline{I}_0\bar h,
\]
where  $ I_0$ is an inner function such that $I_0(0)=0$, and $h\in
H^2$ is outer. This means that $|f(z)|=|h(z)|$ a.e. on $|z|=1$ and
consequently $h(z)=cf(z)$, where $c$ is a unimodular constant.
Replacing $cI_0$ by  $I_0$,   we get
\begin{equation}
\frac{\overline{g}}{g}=\overline{I}_0\frac{\overline{f}}{f}.\label{ker0}
\end{equation}
It then follows
\[
fK_{I}=\ker T_{\frac{\bar g}{g}}= \ker T_{\frac{
\overline{I}_0\overline{f}}{f}}= fK_{I_0}(|f|^2),
\]
which implies $I=I_0$ up to a unimodular constant. Indeed, these
equalities imply that an analytic  function $h$  can be written in
the form $h = f I_0\overline{\psi}_0$, where $\psi_0\in H^2_0(|f|^2)
$, if and only if $h= f I\overline{\psi}$, where $\psi\in H^2_0$.
Since $ |\psi_0|= |\psi|$ a.e. on $|z|=1$ and
$\psi_0\in\mathcal{N}^{+}$, we see that also $\psi_0\in H_0^2$.
Hence $K_{I}=K_{I_0}$.

Consequently, equality  (\ref{ker0}) can be written as
\[
\frac{\bar g}{g}=\frac{\overline{f(1+ I)}}{f(1+ I)} \quad \text{a.e.
on } \partial\D,
\]
which means that the function $\frac{g}{f(1+ I)}$ is real a.e. on
$\partial\D$.  Since this function is in the Smirnov class
$\mathcal{N}^{+}$, our claim follows from Helson's theorem. To prove
the other implication it is enough to observe that if
\[
g=i\frac{I_1+I_2}{I_1-I_2}(1+ I)f,
\]
then
\[
\frac{\bar g}{g}=\frac{\overline{I}{\bar f}}{f}.
\]
{}\hfill$\square$

\section{The complement of $fK_I$ in
$\ker T_{\frac{\bar{I}\bar{f}}{f}}$}

It was noticed in \cite[vol. 2, Cor. 30.21]{FM}  that if  $f$ is an
outer function of the unit norm, $(b,a)$ is  the pair associated
with $f$, and $I$ is an inner function vanishing at the origin that
divides $b$, then
\[
fK_{I}\subset \ker T_{\frac{\bar{I}\bar f}{f}}
\]
and, according to Hayashi's theorem, the equality holds if and only
if the pair $(b_0,a)$ is special and $f_0^2$ is rigid.

Recall that  $\M(a)$ is dense in $\H(b_0)$ if and only if the pair
$(b_0,a)$ is special and $f_0^2$ is a rigid function.

\begin{theorem}
Assume that  $(Ib_0, a)$, where $I$ is inner, and $I(0)=0$, is the
pair associated with an outer function $f$. If the pair $(b_0,a)$ is
not special or the function $f_0^2=\Bigl( \dfrac{a}{1-b_0}\Bigr)^2$
is not rigid,  then for a positive integer $k$,
\[
\dim\bigl(\ker T_{\frac{\bar I \bar f}{f}}\ominus fK_{I}\bigr)=k
\]
if and only if the codimension of $\overline{\M(a)}$ in the de
Branges-Rovnyak space $\H(b_0)$ is $k$.
\end{theorem}

In the proof of this theorem we use some ideas from Sarason's proof
of Hayashi's theorem. If a positive measure $\mu$ on the unit circle
$\partial\D$ is as in (\ref{Herglotz}) and $H^2(\mu)$ is the closure
of the polynomials in $L^2(\mu)$, then  an operator $V_b$ given by
\begin{equation}
(V_bq)(z)=(1-b(z))\int_{\partial\D}\frac{q(e^{i\theta})}{1-e^{-i\theta}z}\,d\mu(e^{i\theta}).
\label{Vb}
\end{equation}
is an isometry of $H^2(\mu)$ onto $\H(b)$ (\cite{Sarason4},
\cite{FM}). Furthermore, if $(b,a)$ is a pair and $f=\frac{a}{1-b}$,
then the operator $T_{1-b}T_{\bar{f}}$ is an isometry of $H^2$ into
$\H(b)$. Its range is all of $\H(b)$ if and only if the pair $(b,a)$
is special.

\medskip

\noindent \textit{Proof of Theorem 2.}  Since the pair $(b,a)$ is
special, the operator $T_{1-b}T_{\bar f}$ is an isometry of $H^2$
onto $\H(b)$. Moreover, since $I$ divides $b$, $T_{f}$ acts as an
isometry on $K_{I}$ and  $T_{1-b}T_{\bar{f}}(fK_I)=K_I$
 \cite{Sarason3}. Hence
\begin{align*}
\H(b)&=T_{1-b}T_{\bar f}(H^2)= T_{1-b}T_{\bar
f}\bigl(\overline{T_{\frac{If}{\bar f}}(H^2)}\oplus
\big(T_{\frac{If}{\bar f}}(H^2)\bigr)^\bot\bigr)\\ &=
\overline{I\M(a)}^b\oplus T_{1-b}T_{\bar f}(\ker T_{\frac{\bar I\bar
f}{f}})\\&= \overline{I\M(a)}^b\oplus T_{1-b}T_{\bar f} (fK_I)
\oplus T_{1-b}T_{\bar f}(\ker T_{\frac{\bar I\bar f}{f}}\ominus
fK_I)
\\&= \overline{I\M(a)}^b\oplus K_I\oplus T_{1-b}T_{\bar
f}(\ker T_{\frac{\bar I\bar f}{f}} \ominus fK_I),
\end{align*}
where $\overline{T_{\frac{If}{\bar f}}(H^2)}$ denotes the closure of
$T_{\frac{If}{\bar f}}(H^2)$ in $H^2$ and $\overline{I\M(a)}^b$
denotes the closure of $I\M(a)$ in $\H(b)$. On the other hand,
\[
\H(b) = \H(b_0I) =K_I\oplus I\H(b_0)= K_I\oplus
I(\H(b_0)\ominus\overline{\M(a)}^{b_0})\oplus
I\overline{\M(a)}^{b_0}.
 \]
Since $T_{I}\colon\, \H(b_0)\to   \H(Ib_0)$ is an isometry
(\cite[Prop. 4]{Sarason3}), $ I\overline{(\M(a))}^{b_0}= \overline{I
\M(a)}^{b}$. It then follows,
\begin{equation}
T_{1-b}T_{\bar f}(\ker T_{\frac{\bar I\bar f}{f}}\ominus fK_I)=
I(\H(b_0)\ominus\overline{\M(a)}^{b_0}).\label{Hb}
\end{equation}
{}\hfill$\square$

\medskip

We remark that the orthogonal complement of $\M(a)$ in $\H(b)$  is
discussed in Section~5.

\section{The Example}

Let, as in the previous sections, $f$ be an outer function  in $H^2$
and let $(b,a)$ be the pair associated with $f$. Let $b=Ib_0$, where
$I$ is an inner function such that $I(0)=0$ and
$f_0=\frac{a}{1-b_0}$. Then $fK_I\subset\ker
T_{\frac{\bar{I}\bar{f}}{f}}$ and equality holds if and only if the
pair $(b_0,a)$ is special and $f_0^2$ is rigid. Moreover, if the
pair $(b_0,a)$ is special and $ f_0^2 $ is rigid, then $(b,a)$ is
special and $ f^2$  is rigid but the  converse implication fails
(\cite[p.158]{Sarason1}).

In \cite[vol. 2, pp. 541--542]{FM} the authors constructed a
function $h$ in $\ker T_{\frac{\bar{I}\bar f}{f}}$ which is not in
$fK_{I}$ under the assumption that $f^2$ is not rigid. Here we
consider the function $f$    such that $f^2$ is rigid, the pair
$(b_0,a)$ is special  but $f^2_0$ is not rigid,  and describe the
space $\ker T_{\frac {\bar I\bar f}{f}} \ominus fK_{I} $.

Our  example is a slight modification of the one given  in \cite{
Sarason2}, see also \cite[vol. 2, p. 494]{FM}.   The corresponding
functions $f$ and $f_0$ are defined by taking $a(z)=\frac{1}{2}(
1+z)$, $b_0(z)=\frac{1}{2}z(1-z)$, and $I(z)=zB(z)$, where $B(z)$ is
a Blaschke produkt with zero sequence $\{r_n\}_{n=1}^{\infty}$ lying
in $(-1,0)$ and converging to $-1$. It has been proved in
\cite{Sarason2} (see also \cite[vol. 2, pp. 494--496]{FM} that $f^2$
is rigid while $f^2_0$ is not. Notice that the pair $(b_0,a)$ is
rational and the point $-1$ is the only zero of the function $a$. It
then follows from \cite[Thm. 4.1.]{LN} (see also \cite{FHR1}) that
$\M(a)$ is a closed subspace of $\H(b_0)$ and
\[
\H(b_0)= \M(a)\oplus {\C}k_{-1}^{b_0},
\]
where
\[
k_{-1}^{b_0}(z)=\frac{1-\overline{b_0(-1)}b_0(z)}{1+z}=\frac{2-z}{2}.
\]
Thus we see that
\[
\H(b_0)\ominus\M(a)= \C k_{-1}^{b_0}.
\]
Moreover  (\ref{Hb}) implies that
\[
T_{1-b}T_{\bar f}(\ker T_{\frac{\bar I\bar f}{f}}\ominus fK_I)= \C I
k_{-1}^{b_0}.
\]
Our aim is to prove that
\begin{equation}
\ker T_{\frac{\bar I\bar f}{f}}\ominus fK_I=\C g,\label{funkcja}
\end{equation}
where the function $g\in H^2$ is given by $g=fk_{-1}(I+1),$  with
$k_{-1}(z)=(1+z)^{-1}$, $z\in\D$.

For $\lambda$ in $\D$ let  $k_\lambda$ denote the kernel function in
$H^2$ for the functional of evaluation at $\lambda$,
$k_\lambda(z)=(1-\bar{\lambda}z)^{-1}$. In the proof of
(\ref{funkcja}) we will apply the following

\begin{lem}[\cite{HSS}]
\begin{enumerate}
\item[{\rm(i)}]
$P_{+}\left(|f|^2Ik_\lambda\right)=\dfrac{Ik_\lambda}{1-b}+\dfrac{\overline{b_0(\lambda)}k_\lambda}{1-\overline{b(\lambda)}}$.
\item[{\rm(ii)}]
$P_{+}\left(|f|^2k_\lambda\right)=\dfrac{k_\lambda}{1-b}+\dfrac{\overline{b(\lambda)}k_\lambda}{1-\overline{b(\lambda)}}$.
\end{enumerate}
\end{lem}

Since $I(r_n)=0$, (i) and (ii) in the Lemma yield
\begin{align*}
T_{1-b}T_{\bar f}(f I k_{r_n})&= Ik_{r_n}(1-\overline{b_0(r_n)}b_0)+
\overline{b_0(r_n)}k_{r_n}, \\
\noalign{\smallskip}
T_{1-b}T_{\bar f}(f  k_{r_n})&=k_{r_n}.
\end{align*}
Hence
\begin{equation}
T_{1-b}T_{\bar f}(f
k_{r_n}(I-\overline{b_0(r_n)}))=Ik_{r_n}(1-\overline{b_0(r_n)}b_0)=Ik_{r_n}^{b_0}.\label{Kr}
\end{equation}
It follows from \cite[vol. 2, Thm. 21.1]{FM} that
\[
\|k_{r_n}^{b_0}-
k_{-1}^{b_0}\|_{b_0}\underset{n\to\infty}{\longrightarrow}0.
\]

Next,  since $T_{I}\colon\, \H(b_0)\to \H(Ib_0)=\H(b)$ is an
isometry and $T_{1-b}T_{\bar f}$ is an isometry of $H^2$ onto
$\H(b)$, we see that $ \{fk_{r_n}(I-\overline{b_0(r_n)}) \}_{n\in
\N}$ is a bounded sequence in $H^2$. So it contains a subsequence
that converges weakly, say, to a function  $g\in H^2$. Without loss
of generality, we may assume that the sequence
$\{fk_{r_n}(I-\overline{b_0(r_n)})\}$ itself converges weakly to $g$. Then
for any point $z\in\D$,
\begin{align*}
g(z)&= \langle g, k_z\rangle= \lim\limits_{n\to\infty}\langle
fk_{r_n}(I-\overline{b_0(r_n)}), k_z\rangle \\
&=
\lim_{n\to\infty}\frac{f(z)(I(z)-\overline{b_0(r_n)})}{1-r_nz}=
 \frac{f(z)(I(z)+1)}{1+z}.
\end{align*}
Now observe that since
\[
\Vert fk_{r_n}(I-\overline{b_0(r_n)})\Vert_{2}=\Vert
k_{r_n}^{b_0}\Vert_{b_0}\quad \text{and}\quad \Vert
fk_{-1}(I+1)\Vert_{2}=\Vert k_{-1}^{b_0}\Vert_{b_0},
\]
$fk_{r_n}(I-\overline{b_0(r_n)})\to fk_{-1}(I+1)$ in $H^2$ strongly.
Finally, passing to the limit in (\ref{Kr}) gives
\[
T_{1-b}T_{\bar f}(fk_{-1}(I+1))=Ik_{-1}(1+b_0)=Ik_{-1}^{b_0},
\]
which proves (\ref{funkcja}).

\begin{rem}
One can check directly that the function $g=fk_{-1}(I+1)$ is in
$\ker T_{\frac{\bar I\bar f}{f}}\ominus fK_I$.

Indeed, we have
\[
T_{\frac {\bar I\bar
f}{f}}\left(fk_{-1}(I+1)\right)=P_{+}\left(\bar{f}\bar{I}\frac{I+1}{1+z}\right)
=P_{+}\left(\bar{f}\frac{\bar{z}(\bar{I}+1)}{\bar{z}+1}\right)=P_{+}\left(\bar{z}\bar{f}\overline{k_{-1}}(\bar{I}+1)\right)=0.
\]
To see that the functions $(fk_{r_n}I-b_0(r_n)fk_{r_n})$ are
orthogonal to $fK_{I}$ note that a function $h\in H^2$ is in $K_{I}$
if and only if $h= h - I P_{+}(\bar Ih)$. So, we have  to check that
for any $h\in H^2$,
\[
\langle fk_{r_n}I-b_0(r_n)fk_{r_n}, f(h-I P_{+}(\bar Ih))\rangle=0.
\]

Since the functions $\{k_{\lambda}, \lambda\in \mathbb{D}\}$ are
dense in $H^2$, it is enough to show that  for  any
$\lambda\in\mathbb D$,
\[
\langle fk_{r_n}I-b_0(r_n)fk_{r_n}, f(k_{\lambda}-I P_{+}(\bar
Ik_{\lambda}))\rangle=\langle fk_{r_n}I-b_0(r_n)fk_{r_n},
f(k_{\lambda}-\overline{I(\lambda)}Ik_{\lambda})\rangle= 0.
\]
Finally, the last equality follows from
\begin{align*}
\langle fk_{r_n}I,fk_{\lambda}\rangle&=
\frac{I(\lambda)k_{r_n}(\lambda)}{1- b(\lambda)}+
b_0(r_n)k_{r_n}(\lambda),\\
\langle fk_{r_n}I,
-\overline{I(\lambda)}fIk_{\lambda}\rangle&=
-\frac{I(\lambda)k_{r_n}(\lambda)}{1-b(\lambda)},\\
 \langle -b_0(r_n)fk_{r_n}, fk_{\lambda}\rangle&=
-\frac{b_0(r_n)k_{r_n}(\lambda)}{1-b(\lambda)},\\
\intertext{ and}
 \langle
-b_0(r_n)fk_{r_n},-\overline{I(\lambda)}fIk_{\lambda}\rangle&=
b_0(r_n)I(\lambda)\frac{b_0(\lambda)k_{r_n}(\lambda)}{1-b(\lambda)}=
\frac{b_0(r_n)b(\lambda)k_{r_n}(\lambda)}{1-b(\lambda)}.
\end{align*}

\end{rem}

\section{A remark on orthogonal complement of $\M(a)$ in $\H(b)$}

In this section we  continue to assume that $b$ is nonextreme. Let
$\H_0(b)$ denote the orthogonal complement of $\M(a)$ in $\H(b)$.
Let $Y$ be the restriction of the shift operator $S$ to $\H(b)$ and
let $Y_0$ be the compression of $Y$ to the subspace $\H_0(b)$.
Necessary and sufficient conditions for  $\H_0(b)$ to have finite
dimension are given in Chapter X of \cite{Sarason4}.  The space
$\H_0(b)$ depends on the spectrum of the restriction of the operator
$Y^\ast$ to $\H_0(b)$ which actually equals $Y^\ast_0$. The spectrum
of $Y^\ast_0$ is contained in the unit circle. More exactly, if
$|z_0|=1$ and $k$ is a positive integer, then
$\ker(Y-\bar{z}_0)^k\subset \H_0(b)$ and the dimension of $ \H_0(b)$
is $N$  if  and only if the operator $Y^\ast_0$ has eigenvalues
$z_1$, $z_2$, \ldots, $z_s$ with their algebraic multiplicities
$n_1$, \ldots, $n_s$  and $N=n_1+n_2+\dots+n_s$.

In this section we consider the case when the eigenspaces
correspondings to eigenvalues $z_1$, $z_2$, \ldots, $z_s$ are one
dimensional.

For $|\lambda|=1$ let $\mu_{\lambda}$ denote the measure for which
equality in (\ref{Herglotz}) holds when $b$ is replaced by
$\bar{\lambda}b$. If we put $F_{\lambda}(z)=
\frac{a}{1-\bar{\lambda}b}$, then the Radon-Nikodym derivative of
the absolutely continuous component of $\mu_{\lambda}$ is
$|F_{\lambda}|^2$.

The following theorem is proved in Chapter X of \cite{Sarason4}.
\begin{theoremS}
Let $z_0$ be a point of $\partial\D$ and $\lambda$ a point of
$\partial\D$ such that the measure $\mu_\lambda$ is absolutely
continuous. The following conditions are equivalent.
\begin{enumerate}
\item[{\rm(i)}]
$\bar{z}_0$ is an eigenvalue of $Y^\ast$.
\item[{\rm(ii)}]
The function $\dfrac{F_\lambda(z)}{1-\bar{z}_0z}$ is in $H^2$.
\item[{\rm(iii)}]
The function $b$ has an angular derivative in the sense of
Carath\'eodory at $z_0$.
\end{enumerate}
\end{theoremS}

The space $\H_0(b)$ is described by means of an operator $A_\lambda$
on $H^2$ that intertwines
$T_{1-\bar{\lambda}b}T_{\overline{F}_\lambda}$ with the operator
$Y^\ast$, i.e.
\begin{equation}
T_{1-\bar{\lambda}b}T_{\overline{F}_\lambda}A_\lambda=Y^\ast
T_{1-\bar{\lambda}b}T_{\overline{F}_\lambda}.\label{inter}
\end{equation}
The operator $A_\lambda$ is given by
\[
A_\lambda=S^\ast-F_\lambda(0)^{-1}(S^\ast F_\lambda\otimes 1).
\]

It follows from the proof of Sarason's theorem that if one of
conditions (i) -- (iii) holds true, then the space
$\ker(A_\lambda-\bar{z}_0)$ is one dimensional and is spanned by the
function
\begin{equation}
g(z)=\frac {F_{\lambda}(z)}{1-\overline{z}_0z}=
F_{\lambda}(z)k_{z_0}(z).\label{functiong}
\end{equation}
We also  note that since $b$ has  the angular derivative in the
sense of Carath\'{e}odory at $z_0$, the function
\[
 k_{z_0}^b(z)=
\dfrac{1-\overline{b(z_0)}b(z)}{1-\bar{z_0}z},
\]
where $b(z_0) $ is the nontangential limit of $b$ in $z_0$, is in
$\H(b)$.

Using Sarason's theorem cited above  we show the following
\begin{theorem}
If   the assumptions of  Sarason's theorem are
satisfied and $\bar{z}_0$ is an eigenvalue of $Y^\ast$, then
$\ker(Y^\ast-\bar{z_0})$ is spanned by $k_{z_0}^b$.
\end{theorem}

\begin{proof}
In view of (\ref{functiong}) and (\ref{inter}) the space
$\ker(Y^\ast-\bar{z_0})$ is spanned by
\[
h= T_{1-\overline{\lambda}b}T_{\overline{F}_{\lambda}}g=
T_{1-\overline{\lambda}b}T_{\overline
{F}_{\lambda}}(F_{\lambda}k_{z_0}).
\]

It is known that if $V_b$ is given by (\ref{Vb}), then for a fixed
$w\in\D$,
\[
V_b((1-\overline{b(w)})k_w)=k^b_w.
\]

Hence we get
\begin{align*}
V_{\bar{\lambda}b}((1-\lambda\overline {b(w)})k_{w})(z)&=
(1-\bar{\lambda}b(z))( 1-\lambda\overline{b(w)}) \int_{\partial
\D}\frac{|F_{\lambda}(e^{i\theta})|^2d\theta}{(1-\bar
we^{i\theta})(1-ze^{-i\theta})}\\ &=
(1-\bar{\lambda}b(z))T_{\overline{F}_{\lambda}}((1-\lambda\overline
{b(w)})F_{\lambda}k_w)(z)=k_{w}^b(z).
\end{align*}

Let $\{z_n\}$ be a sequence in $\mathbb D$ converging
nontangentially to $z_0$. Then
\[
T_{1-\overline{\lambda}b}T_{\overline
{F}_{\lambda}}((1-\lambda\overline{b(z_n)})F_{\lambda}k_{z_n})=k_{z_n}^b.
\]
Observe also that since $\mu_{\lambda}$ is absolutely continuous,
$b(z_0)\ne \lambda$ \cite[VI-7, VI-9]{Sarason4}. Moreover, we know
that $k_z^b$ tends to $k_{z_0}^b$ weakly and $\|k_{z}^b\|_b$ tends
to $\|k_{z_0}^b\|_b$ as $z$ tends nontangentially to $z_0$
\cite[VI-5]{Sarason4}. Clearly this implies that $k_z^b$ tends to
$k_{z_0}^b$ in norm as $z$ tends to $z_0$ nontagentially. It then
follows that the sequence
$\{(1-\lambda\overline{b(z_n)})F_{\lambda}k_{z_n}\}$ converges in
$H^2$, which in turn implies compact  and pointwise convergence.
Hence passing to the limit in the last equality yields
\[
T_{1-\overline{\lambda}b}T_{\overline
{F}_{\lambda}}(F_{\lambda}k_{z_0})= Ck_{z_0}^b,
\]
where $C=(1-\lambda\overline{b(z_0)})^{-1}$.
\end{proof}

Finally we remark that the next Corollary generalizes results
obtained in \cite{FHR1} and in \cite{LN} for the case when pairs
$(b,a)$ are rational.

\begin{wn}
If $z_1$, $z_2$, \ldots, $z_s$ are the only eigenvalues of $Y_0$ and
each of them is of multiplicity one, then $\H_0(b)$ is spanned by
the function $k_{z_1}^b$, $k_{z_2}^b$, \ldots, $k_{z_s}^b$.
\end{wn}

\medskip

\noindent
{\bf Acknowledgements.}
The third named author would like to thank the Institute of Mathematics of the Maria Curie-Sk{\l}odowska University for supporting his visit to Lublin where part of this paper was written.

\end{document}